\documentclass{amsart}

\usepackage{amsmath, amssymb, amsthm}
\usepackage{color}

\newtheorem{theorem}{Theorem}[section]

\newtheorem{corollary}[theorem]{Corollary}
\newtheorem{claim}[theorem]{Claim}

\newtheorem{question}[theorem]{Question}

\newtheorem{problem}[theorem]{Problem}

\theoremstyle{definition}
\newtheorem{definition}[theorem]{Definition}
\newtheorem{remark}[theorem]{Remark}

\newcommand{\cf}{\mathrm{cf}}
\newcommand{\dom}{\mathrm{dom}}
\newcommand{\bb}{\mathbb}

\newcommand{\otp}{\mathrm{otp}}

\newcommand{\acc}{\mathrm{acc}}

\title{A note on highly connected and well-connected Ramsey theory}
\author{Chris Lambie-Hanson}
\address{Institute of Mathematics of the Czech Academy of Sciences \\ 
\v{Z}itn\'{a} 25, Prague 1, 115 67, Czech Republic}
\email{lambiehanson@math.cas.cz}
\urladdr{http://math.cas.cz/lambiehanson}
\subjclass[2020]{03E02, 03E35, 05C63}
\keywords{partition relations, highly connected, well-connected, guessing models,
square principles}
\thanks{We thank Jeffrey Bergfalk and Michael Hru\v{s}\'{a}k for reading an early
draft and providing helpful comments and corrections.}

\begin{document}
\begin{abstract}
  We study a pair of weakenings of the classical partition relation
  $\nu \rightarrow (\mu)^2_\lambda$ recently introduced by Bergfalk-Hru\v{s}\'{a}k-Shelah
  and Bergfalk, respectively. Given an edge-coloring of the complete graph on
  $\nu$-many vertices, these weakenings assert the existence of monochromatic
  subgraphs exhibiting high degrees of connectedness rather than the existence of
  complete monochromatic subgraphs asserted by the classical relations. As a result,
  versions of these weakenings can consistently hold at accessible cardinals
  where their classical analogues would necessarily fail. We prove some
  complementary positive and negative results indicating the effect of large
  cardinals, forcing axioms, and square principles on these partition relations.
  We also prove a consistency result indicating that a non-trivial instance of
  the stronger of these two partition relations can hold at the continuum.
\end{abstract}
\maketitle

\section{Introduction}

In this paper, we study some natural variations of the classical partition
relation for pairs. Recalling the arrow notation of Erd\H{o}s and Rado,
given cardinals $\mu$, $\nu$, and $\lambda$, the expression
\[
  \nu \rightarrow (\mu)^2_\lambda
\]
denotes the assertion that, for every coloring $c:[\nu]^2 \rightarrow \lambda$,
there is a set $X \subseteq \nu$ of cardinality $\mu$
such that $c \restriction [X]^2$ is constant. This can be usefully interpreted
in the language of graph theory as asserting that for every edge-coloring of
the complete graph $K_\nu$ on $\nu$-many vertices using $\lambda$-many colors,
there is a monochromatic subgraph of $K_\nu$ isomorphic to $K_\mu$.
With this notation, the infinite Ramsey theorem for pairs can be expressed succinctly as
\[
  \aleph_0 \rightarrow (\aleph_0)^2_k
\]
for all $k < \omega$.

When one tries to generalize the infinite Ramsey theorem to uncountable cardinals
in the most straightforward way, by replacing each instance of `$\aleph_0$' in
the above expression by some fixed uncountable cardinal `$\nu$', one immediately
runs into a statement that can only hold at large cardinals, as the assertion
that $\nu$ is uncountable and $\nu \rightarrow (\nu)^2_2$ is equivalent to the
assertion that $\nu$ is weakly compact.

To achieve a consistent statement at accessible uncountable cardinals, then,
one must weaken the statement $\nu \rightarrow (\nu)^2_2$. One natural approach
is to increase the value of the cardinal on the left side of the expression.
This is precisely what was done by Erd\H{o}s and Rado in \cite{erdos_rado}.
A special case of what became known as the Erd\H{o}s-Rado theorem states that,
for every infinite cardinal $\kappa$, we have
\[
  (2^\kappa)^+ \rightarrow (\kappa^+)^2_\kappa.
\]
This result is sharp, in the sense that there are colorings witnessing the
negative relations $2^\kappa \not\rightarrow (3)^2_\kappa$ and
$2^\kappa \not\rightarrow (\kappa^+)^2_2$. (We refer the reader to
\cite[\S 7]{kanamori} for proofs and further discussion of these facts.)

Another approach is to weaken the meaning of the arrow notation, in particular
by replacing the requirement that the monochromatic subgraph witnessing the
partition relation be complete by a weaker but still nontrivial notion of
largeness. This is the approach taken by Bergfalk-Hru\v{s}\'{a}k-Shelah and
Bergfalk in \cite{highly_connected} and \cite{well_connected}, where they introduce
the notions of partition relations for \emph{highly connected} and \emph{well-connected}
subsets, respectively. It is these notions that provide the primary subject of
this paper; let us turn to their definitions, beginning with the partition relation for
\emph{highly connected} subsets.

Throughout the paper, if $G = (X, E)$ is a graph and $Y \subseteq X$, then we
let $G \setminus Y$ denote the graph $(X \setminus Y, ~ E \cap [X \setminus Y]^2)$.
We will also sometime write $|G|$ to mean $|X|$.

\begin{definition}
  Given a graph $G = (X, E)$ and a cardinal $\kappa$, we say that $G$ is
  \emph{$\kappa$-connected} if $G \setminus Y$ is
  connected for every $Y \in [X]^{<\kappa}$. We say that
  $G$ is \emph{highly connected} if it is $|G|$-connected.
\end{definition}

\begin{definition}(Bergfalk-Hru\v{s}\'{a}k-Shelah \cite{highly_connected})
  \label{hc_relation_def}
  Suppose that $\mu$, $\nu$, and $\lambda$ are cardinals. The partition
  relation $\nu \rightarrow_{hc} (\mu)^2_\lambda$ is the assertion that,
  for every coloring $c:[\nu]^2 \rightarrow \lambda$, there is an
  $X \in [\nu]^\mu$ and a highly connected subgraph $(X, E)$ of
  $(\nu, [\nu]^2)$ such that $c \restriction E$ is constant.
\end{definition}

As was noted in \cite{highly_connected}, a finite graph is highly connected
if and only if it is complete. As a result, if $\mu$ is finite, then the
relation $\nu \rightarrow_{hc} (\mu)^2_\lambda$ is simply the classical relation
$\nu \rightarrow (\mu)^2_\lambda$. In the context of infinite $\mu$, however,
the $\rightarrow_{hc}$ version of a relation can consistently hold in situations
where the classical version necessarily fails. For example, the following is
proven in \cite{highly_connected}.

\begin{theorem}[Bergfalk-Hru\v{s}\'{a}k-Shelah \cite{highly_connected}]
  It is consistent, relative to the consistency of a weakly compact cardinal,
  that
  \[
    2^{\aleph_1} \rightarrow_{hc} \left(2^{\aleph_1}\right)^2_{\aleph_0}.
  \]
\end{theorem}

We now recall the partition relation for \emph{well-connected} subsets.

\begin{definition}(Bergfalk \cite{well_connected}) \label{wc_def}
  Suppose that $\mu$, $\nu$, and $\lambda$ are cardinals. Given a coloring
  $c:[\nu]^2 \rightarrow \lambda$ and a fixed color $i < \lambda$, we say
  that a subset $X \subseteq \nu$ is \emph{well-connected in color $i$}
  (with respect to $c$) if, for every $\alpha < \beta$ in $X$, there is a finite path
  $\langle \alpha_0, \ldots, \alpha_n \rangle$ (not necessarily contained
  in $X$) such that
  \begin{itemize}
    \item $\alpha_0 = \alpha$ and $\alpha_n = \beta$;
    \item $\alpha_k \geq \alpha$ for all $k \leq n$; and
    \item $c(\alpha_k, \alpha_{k+1}) = i$ for all $k < n$.
  \end{itemize}
  The partition relation $\nu \rightarrow_{wc} (\mu)^2_\lambda$ is
  the assertion that, for every coloring $c:[\nu]^2 \rightarrow \lambda$,
  there are $X \in [\nu]^\mu$ and $i < \lambda$ such that $X$ is well-connected
  in color $i$.
\end{definition}

As usual, $\nu \not\rightarrow_{hc} (\mu)^2_\lambda$ and $\nu \not\rightarrow_{wc}
(\mu)^2_\lambda$ denote the negations of the respective partition relations.

The $\rightarrow_{wc}$ relation is a clear weakening of the classical $\rightarrow$
relation and is in fact a weakening of $\rightarrow_{hc}$. Indeed, by \cite[Lemma
6]{well_connected}, given cardinals $\mu$, $\nu$, and $\lambda$, we have
\[
  \left(\nu \rightarrow (\mu)^2_\lambda \right) \Rightarrow
  \left(\nu \rightarrow_{hc} (\mu)^2_\lambda \right) \Rightarrow
  \left(\nu \rightarrow_{wc} (\mu)^2_\lambda \right).
\]
We have already seen that the left implication in the above statement can consistently
fail. The right implication can consistently fail as well. For example, the following
is proven in \cite{well_connected}.

\begin{theorem}[Bergfalk \cite{well_connected}]
  It is consistent, relative to the consistency of a weakly compact cardinal,
  that $\aleph_2 \rightarrow_{wc} (\aleph_2)^2_{\aleph_0}$ holds but
  $\aleph_2 \rightarrow_{hc} (\aleph_2)^2_{\aleph_0}$ fails.
\end{theorem}

The model for the above theorem is Mitchell's model for the tree property at $\aleph_2$
from \cite{mitchell}. The question of the consistency of $\aleph_2 \rightarrow_{hc}
(\aleph_2)^2_{\aleph_0}$ was asked in \cite{highly_connected} and remains open.

In this paper, we prove some further results about the relations $\rightarrow_{hc}$
and $\rightarrow_{wc}$, some of them addressing questions from \cite{highly_connected}
and \cite{well_connected}. We provide here a brief outline of the remainder of the
paper. In Section \ref{positive_section}, we isolate some instances in which
the relations $\nu \rightarrow_{hc} (\nu)^2_\lambda$ and $\nu \rightarrow_{wc}
(\nu)^2_\lambda$ necessarily hold due either to the existence of certain large
cardinals or to certain forcing axioms holding. In particular, we show that, if $\mathsf{PFA}$
holds, then $\nu \rightarrow_{wc} (\nu)^2_{\aleph_0}$ holds for every
regular cardinal $\nu \geq \aleph_2$. This is optimal in the sense that
it was proven in \cite{well_connected} that $\aleph_1 \not\rightarrow_{wc}
(3)^2_{\aleph_0}$ In Section \ref{negative_section}, we prove
some complementary negative results, in particular indicating that certain
square principles imply the failure of instances of $\nu \rightarrow_{wc} (\nu)^2_\lambda$.
In the process, we introduce natural square-bracket versions of the partition
relations for highly connected and well-connected subsets. In Section
\ref{forcing_sec}, we prove the consistency, relative to the consistency of
a weakly compact cardinal, of the partition relation
$2^{\aleph_0} \rightarrow_{hc} [2^{\aleph_0}]^2_{\aleph_0, 2}$.
This is sharp due to the fact that, as shown in \cite{highly_connected}, the
negative relation $2^{\aleph_0} \not\rightarrow_{hc} (2^{\aleph_0})^2_{\aleph_0}$ is provable in
$\mathrm{ZFC}$. Finally, in Section \ref{questions_sec}, we point out some recent work 
solving open questions stated in the first draft of this paper.

\subsection*{Notation and conventions} If $X$ is a set and $\mu$ is a cardinal,
then $[X]^\mu = \{Y \subseteq X \mid |Y| = \mu\}$. A \emph{graph} is a pair
$G = (X, E)$, where $X$ is a set and $E \subseteq [X]^2$. Given a function $c$
with domain $[X]^2$, we often slightly abuse notation and write, for instance,
$c(a,b)$ instead of $c(\{a,b\})$. While elements of $[X]^2$ are unordered pairs,
we will sometimes care about the relative order of the elements of such a pair.
In particular, if $X$ is a set of ordinals, then we will use the notation
$(\alpha, \beta) \in [X]^2$ to indicate the conjunction of the two statements
$\{\alpha, \beta\} \in [X]^2$ and $\alpha < \beta$.

A \emph{path} in a graph $G = (X, E)$ is a finite sequence $\langle x_0, \ldots,
x_n \rangle$ of pairwise distinct elements of $X$ such that
$\{x_i, x_{i + 1}\} \in E$ for all $i < n$. We say that $G$ is \emph{connected}
if, for all distinct $x, y \in X$, there is a path $\langle x_0, \ldots, x_n \rangle$
in $G$ such that $x_0 = x$ and $x_n = y$.

If $x$ is a set of ordinals, then the set of \emph{accumulation points} of $x$,
denoted by $\acc(x)$, is defined to be $\{\alpha \in x \mid \sup(x \cap \alpha) = \alpha\}$.
In particular, if $\beta$ is an ordinal, then $\acc(\beta)$ is the set of limit
ordinals below $\beta$.

\section{Positive results} \label{positive_section}

In \cite{well_connected}, Bergfalk asked about conditions under which the
relations $\mu^+ \rightarrow_{hc} (\mu)^2_{\cf(\mu)}$ and $\mu^+
\rightarrow_{wc} (\mu)^2_{\cf(\mu)}$ hold, where $\mu$ is a singular cardinal.
In this section, we provide two scenarios in which such relations
(and more) hold. The first scenario simply involves the presence of large
cardinals and produces instances of the the highly connected partition relation.
The second involves the existence of guessing models and produces instances of
the well-connected partition relation. Complementary negative results appear in
the next section.

\begin{definition}
  Suppose that $\theta \leq \kappa$ are regular, uncountable cardinals.
  $\kappa$ is \emph{$\theta$-strongly compact} if, for every set $X$ and every
  $\kappa$-complete filter $F$ over $X$, $F$ can be extended to a $\theta$-complete
  ultrafilter over $X$. $\kappa$ is \emph{strongly compact} if it is $\kappa$-strongly
  compact.
\end{definition}

\begin{theorem}
  Suppose that $\theta \leq \kappa$ are regular uncountable cardinals
  and $\kappa$ is $\theta$-strongly compact. Suppose moreover that
  $\lambda$ and $\nu$ are cardinals with $\lambda < \theta$ and
  $\cf(\nu) \geq \kappa$. Then
  \[
    \nu \rightarrow_{hc} (\nu)^2_\lambda.
  \]
\end{theorem}

\begin{proof}
  Fix a coloring $c : [\nu]^2 \rightarrow \lambda$. We must find a color
  $i < \lambda$ and a set $X \in [\nu]^\nu$ such that the graph
  $(X, c^{-1}(i) \cap [X]^2)$ is highly connected.

  Using the fact that $\kappa$ is $\theta$-strongly compact, let $U$ be a
  $\theta$-complete ultrafilter over $\nu$ extending the $\kappa$-complete
  filter
  \[
    F = \{X \subseteq \nu \mid |\nu \setminus X| < \nu \}.
  \]
  In particular, every element of $U$ has cardinality $\nu$. Now, for each
  $\alpha < \nu$, use the $\theta$-completeness of $U$ to find a color
  $i_\alpha < \lambda$ and a set $X_\alpha \subseteq \nu \setminus (\alpha + 1)$
  such that $X_\alpha \in U$ and $c(\alpha, \beta) = i_\alpha$ for all
  $\beta \in X_\alpha$. Now use the $\theta$-completeness of $U$ again to find
  a fixed color $i < \lambda$ and a set $X \in U$ such that $i_\alpha = i$ for
  all $\alpha \in X$.

  Let $G = (X, c^{-1}(i) \cap [X]^2)$. We claim that $X$ is highly connected.
  To this end, let $Y \in [X]^{<\nu}$, and fix $\alpha < \beta$ in
  $X \setminus Y$. To show that $\alpha$ and $\beta$ are connected in $G \setminus Y$, find
  \[
    \gamma \in (X \cap X_\alpha \cap X_\beta) \setminus Y,
  \]
  and note that $c(\alpha, \gamma) = c(\beta, \gamma) = i$. Then
  $\langle \alpha, \gamma, \beta \rangle$ is a path from $\alpha$ to $\beta$
  in $G \setminus Y$.
\end{proof}

\begin{corollary}
  Suppose that $\mu$ is a singular limit of strongly compact cardinals.
  Then
  \[
    \mu^+ \rightarrow_{hc} (\mu^+)^2_\lambda
  \]
  for all $\lambda < \mu$.
\end{corollary}

Recall the following definition, which comes from \cite{viale_guessing_models}
and is a generalization of a notion from \cite{viale_weiss}.

\begin{definition}
  Suppose that $\kappa < \chi$ are regular uncountable cardinals,
  $M \prec (H(\chi), \in)$, and $\kappa \leq |M| < \chi$.
  \begin{enumerate}
    \item Suppose that $x \in M$ and $d \subseteq x$.
    \begin{enumerate}
      \item We say that $d$ is \emph{$(\kappa, M)$-approximated} if $d \cap z \in M$
      for every $z \in M \cap \mathcal{P}_\kappa(H(\chi))$.
      \item We say that $d$ is \emph{$M$-guessed} if there is $e \in M$ such that
      $e \cap M = d \cap M$.
    \end{enumerate}
    \item $M$ is \emph{$\kappa$-guessing} if every $(\kappa, M)$-approximated
    set is $M$-guessed.
  \end{enumerate}
\end{definition}

Given an infinite regular cardinal $\theta$, let $(T_\theta)$ be
the statement asserting that there are arbitrarily large regular cardinals
$\chi$ such that the set
\[
  \{M \prec (H(\chi), \in) \mid |M| = \theta^+, ~ {^{<\theta} M} \subseteq M,
  \text{ and } M \text{ is } \theta^+\text{-guessing}\}
\]
is stationary in $\mathcal{P}_{\theta^{++}}(H(\chi))$.

It is proven by Viale and Wei\ss\ in \cite{viale_weiss}
that the Proper Forcing Axiom
($\mathsf{PFA}$) implies $(T_{\aleph_0})$. Trang, in \cite{trang_guessing_models},
proves the consistency of $(T_{\aleph_1})$, assuming the consistency of a
supercompact cardinal. The proof uses a Mitchell-type forcing construction and
is easily modified to show that, if $\theta$ is a regular cardinal and there
is a supercompact cardinal above $\theta$, then there is a $\theta$-closed
forcing extension in which $(T_{\theta^{+}})$ holds.

\begin{theorem}
  Suppose that $\theta$ is a regular cardinal and $(T_\theta)$
  holds. Then, for every regular cardinal $\nu > \theta^+$ and every
  $\lambda \leq \theta$, we have
  \[
    \nu \rightarrow_{wc} (\nu)^2_\lambda.
  \]
\end{theorem}

\begin{proof}
  Fix a regular cardinal $\nu > \theta^+$, a cardinal $\lambda \leq \theta$, and
  a coloring $c:[\nu]^2 \rightarrow \lambda$. We must find a color $i < \lambda$
  and a set $X \in [\nu]^\nu$ such that $X$ is well-connected in color $i$.
  Given a color $i < \lambda$ and $\alpha < \beta < \nu$, we say that
  $\alpha <_i \beta$ if $\{\alpha, \beta\}$ is well-connected in color $i$.
  It is easily verified that, for every $i < \lambda$, $(\nu, \leq_i)$ is a
  tree, i.e., it is a partial order and, for all $\beta \in \nu$, the set of
  $<_i$-predecessors of $\beta$ is well-ordered by $<_i$ (cf.\
  \cite[Lemma 12]{well_connected}).

  Using the fact that $(T_\theta)$ holds, we may fix a regular cardinal
  $\chi >> \nu$, a well-ordering $\vartriangleleft$ of $H(\chi)$, and an
  elementary submodel $M \prec (H(\chi), \in, \vartriangleleft, \lambda,
  \theta, \nu, c)$ such that
  \begin{itemize}
    \item $|M| = \theta^+$;
    \item ${^{<\theta}M} \subseteq M$; and
    \item $M$ is $\theta^+$-guessing.
  \end{itemize}
  Let $\nu_M = \sup(M \cap \nu)$.

  \begin{claim}
    $\cf(\nu_M) = \theta^+$.
  \end{claim}

  \begin{proof}
    Since $|M| = \theta^+$, we clearly have $\cf(\nu_M) \leq \theta^+$.
    Suppose for sake of contradiction that $\cf(\nu_M) = \mu < \theta^+$.
    Let $A$ be a cofinal subset of $M \cap \nu_M$ of order type $\mu$, and let
    $\langle \alpha_\eta \mid \eta < \mu\rangle$ be the increasing enumeration of $A$.

    We first show that $A$ is $(\theta^+, M)$-approximated. To this end,
    fix $z \in M \cap \mathcal{P}_{\theta^+}(H(\chi))$. Since $A \subseteq
    \nu$, we may assume that $z \subseteq \nu$. Since $|z| < \theta^+$ and
    $\nu > \theta^+$ is regular, we know that $z$ is bounded below $\nu$. Since
    $z \in M$, it follows by elementarity that $z$ is bounded below
    $\nu_M$, so there is $\xi < \mu$ such that $z \subseteq \alpha_\xi$.
    But then $A \cap z \subseteq \{\alpha_\eta \mid \eta < \xi\}$, and
    hence $A \cap z$ is a subset of $M$ and $|A \cap z| < \mu \leq \theta$.
    Since ${^{< \theta} M} \subseteq M$, it follows that $A \cap z \in M$.
    Since $z$ was arbitrary, we have shown that $A$ is $(\theta^+, M)$-approximated.

    Since $M$ is $\theta^+$-guessing, there is $B \in M$ such that $B \cap M =
    A \cap M = A$. Since $A$ is unbounded in $M \cap \nu$, it follows by elementarity
    that $B$ is unbounded in $\nu$, and hence $\otp(B) = \nu$. Let
    $\pi:\nu \rightarrow B$ be the order-preserving
    bijection, and note that $\pi \in M$. Since ${^{<\theta}M} \subseteq M$
    and $\theta \in M$, we have $\theta + 1 \subseteq M$, and in particular
    $\mu + 1 \subseteq M$. But then $\pi``(\mu+1) \subseteq B \cap M = A \cap M$,
    contradicting the fact that $\otp(A) = \mu$.
  \end{proof}

  Since $\cf(\nu_M) = \theta^+ > \lambda$, there is fixed color $i < \theta$
  and an unbounded set $d_0 \subseteq M \cap \nu_M$ such that
  $c(\beta, \nu_M) = i$ for all $\beta \in d_0$. Note that, for all 
  $\alpha < \beta$, both in $d_0$, we have $\alpha <_i \beta$, as witnessed 
  by the path $\langle \alpha, \nu_M, \beta \rangle$. Let $d$ be the
  $<_i$-downward closure of $d_0$, i.e.,
  \[
    d = \{\alpha < \nu \mid \text{there is } \beta \in d_0 \text{ such that }
    \alpha <_i \beta\}.
  \]
  (Note that it may not be the case that $d \subseteq M$.) We claim that
  $d$ is $(\theta^+, M)$-approximated. To see this, fix $z \in M \cap
  \mathcal{P}_{\theta^+}(H(\chi))$. As in the proof of the claim, we may
  assume that $z \subseteq \nu$, and again as in the proof of the claim
  it follows that there is $\beta \in d_0$ such that $z \subseteq \beta$.
  Then, using the fact that $<_i$ is a tree ordering, we have
  \[
    d \cap z = \{\alpha \in z \mid \alpha <_i \beta\}.
  \]
  Since everything needed to define this latter set is in $M$, we have
  $d \cap z \in M$. Since $z$ was arbitrary, we have shown that $d$ is
  $(\theta^+, M)$-approximated.

  As $M$ is $\theta^+$-guessing, there is $e \in M$ such that
  $e \cap M = d \cap M$. By elementarity, it follows that $e$ is a
  cofinal subset of $\nu$ that is linearly ordered by $<_i$. In other
  words, $e$ is well-connected in color $i$. Since $\nu$ is regular,
  we have $|e| = \nu$, so we have proven our theorem.
\end{proof}

\begin{corollary}
  Suppose that $\mathsf{PFA}$ holds. Then, for every regular cardinal
  $\nu \geq \aleph_2$, we have
  \[
    \nu \rightarrow_{wc} (\nu)^2_{\aleph_0}.
  \]
\end{corollary}

\section{Indexed squares and negative results} \label{negative_section}

In this section, we use square principles to isolate situations in which
instances of $\rightarrow_{wc}$ necessarily fail. The results in this section
are refinements of \cite[Lemma 9]{well_connected}. In order to fully state
our results, we introduce \emph{square bracket} versions of the partition
relations being studied.

\begin{definition} \label{well_connected_set_def}
  Suppose that $\nu$ and $\lambda$ are cardinals. Given a coloring $c: [\nu]^k \rightarrow
  \lambda$ and a collection of colors $\Lambda \subseteq \lambda$, we say that a subset
  $X \subseteq \nu$ is \emph{well-connected in $\Lambda$} (with respect to
  $c$) if, for every $\alpha < \beta$ in $X$, there is a finite path
  $\langle \alpha_0, \ldots, \alpha_n \rangle$ (not necessarily contained in $X$)
  such that
  \begin{itemize}
    \item $\alpha_0 = \alpha$ and $\alpha_n = \beta$;
    \item $\alpha_k \geq \alpha$ for all $k \leq n$; and
    \item $c(\alpha_k, \alpha_{k + 1}) \in \Lambda$ for all $k < n$.
  \end{itemize}
\end{definition}

\begin{definition}
  Suppose that $\mu$, $\nu$, $\lambda$, and $\kappa$ are cardinals.
  \begin{enumerate}
    \item The partition relation $\nu \rightarrow_{hc} [\mu]^2_{\lambda, \kappa}$
    (\emph{resp.\ }$\nu \rightarrow_{hc} [\mu]^2_{\lambda, {<}\kappa}$) is
    the assertion that, for every coloring $c:[\nu]^2 \rightarrow \lambda$,
    there is an $X \in [\nu]^\mu$, a highly connected subgraph
    $(X, E)$ of $(\nu, [\nu]^2)$, and a set $\Lambda \in [\lambda]^{{\leq}\kappa}$
    (\emph{resp.\ }$\Lambda \in [\lambda]^{{<}\kappa}$) such that
    $c``E \subseteq \Lambda$.
    \item The partition relation $\nu \rightarrow_{wc} [\mu]^2_{\lambda, \kappa}$
    (\emph{resp.\ }$\nu \rightarrow_{wc} [\mu]^2_{\lambda, {<}\kappa}$) is the
    assertion that, for every coloring $c:[\nu]^2 \rightarrow \lambda$, there
    are $X \in [\nu]^\mu$ and $\Lambda \in [\lambda]^{{\leq}\kappa}$
    (\emph{resp.\ }$\Lambda \in [\lambda]^{{<}\kappa}$) such that $X$ is
    well-connected in $\Lambda$.
  \end{enumerate}
\end{definition}

As usual, the negations of these partition relations will be denoted by, e.g.,
$\nu \not\rightarrow_{hc} [\mu]^2_{\lambda, \kappa}$.
We now recall certain square principles known as \emph{indexed square principles}.

\begin{definition}[Cummings-Foreman-Magidor \cite{cfm}]
  Suppose that $\mu$ is a singular cardinal. A $\square^{\mathrm{ind}}_{\mu,
  \cf(\mu)}$-sequence is a sequence $\langle C_{\alpha, i} \mid
  \alpha \in \acc(\mu^+), ~ i(\alpha) \leq i < \cf(\mu) \rangle$ such that
  \begin{enumerate}
    \item for all $\alpha \in \acc(\mu^+)$, we have $i(\alpha) < \cf(\mu)$;
    \item for all $\alpha \in \acc(\mu^+)$ and all $i(\alpha) \leq i < \cf(\mu)$,
    $C_{\alpha, i}$ is a club in $\alpha$;
    \item for all $\alpha \in \acc(\mu^+)$ and all $i(\alpha) \leq i
    < j < \cf(\mu)$, we have $C_{\alpha, i} \subseteq C_{\alpha, j}$;
    \item for all $\alpha < \beta$ in $\acc(\mu^+)$ and all $i(\beta) \leq i <
    \cf(\mu)$, if $\alpha \in \acc(C_{\beta, i})$, then $i(\alpha) \leq i$
    and $C_{\alpha, i} = C_{\beta, i} \cap \alpha$;
    \item for all $\alpha < \beta$ in $\acc(\mu^+)$, there is $i$ such that
    $i(\beta) \leq i < \cf(\mu)$ and $\alpha \in \acc(C_{\beta, i})$;
    \item there is an increasing sequence $\langle \mu_i \mid i <
    \cf(\mu) \rangle$ of regular cardinals such that
    \begin{enumerate}
      \item $\sup(\{\mu_i \mid i < \cf(\mu)\}) = \mu$;
      \item for all $\alpha \in \acc(\mu^+)$ and all $i(\alpha) \leq i <
      \cf(\mu)$, we have $|C_{\alpha, i}| < \mu_i$.
    \end{enumerate}
  \end{enumerate}
  $\square^{\mathrm{ind}}_{\mu, \cf(\mu)}$ is the assertion that there is a
  $\square^{\mathrm{ind}}_{\mu, \cf(\mu)}$-sequence.
\end{definition}

\begin{definition}[\cite{narrow_systems}] \label{indexed_square_mu_lambda_def}
  Suppose that $\lambda < \mu$ are infinite regular cardinals. Then a $\square^{\mathrm{ind}}
  (\mu, \lambda)$-sequence is a sequence $\langle C_{\alpha, i} \mid \alpha \in
  \acc(\mu), ~ i(\alpha) \leq i < \lambda \rangle$ such that
  \begin{enumerate}
    \item for all $\alpha \in \acc(\mu)$, we have $i(\alpha) < \lambda$;
    \item for all $\alpha \in \acc(\mu)$ and all $i(\alpha) \leq i < \lambda$,
    $C_{\alpha, i}$ is a club in $\alpha$;
    \item for all $\alpha \in \acc(\mu)$ and all $i(\alpha) \leq i < j < \lambda$,
    we have $C_{\alpha, i} \subseteq C_{\alpha, j}$;
    \item for all $\alpha < \beta$ in $\acc(\mu)$ and all $i(\beta) \leq i < \lambda$,
    if $\alpha \in \acc(C_{\beta, i})$, then $i(\alpha) \leq i$ and
    $C_{\alpha, i} = C_{\beta, i} \cap \alpha$;
    \item for all $\alpha < \beta$ in $\acc(\mu)$, there is $i$ such that
    $i(\beta) \leq i < \lambda$ and $\alpha \in \acc(C_{\beta, i})$;
    \item for every club $D$ in $\mu$ and every $i < \lambda$, there is
    $\alpha \in \acc(D)$ such that $D \cap \alpha \neq C_{\alpha, i}$.
  \end{enumerate}
  $\square^{\mathrm{ind}}(\mu, \lambda)$ is the assertion that there is a
  $\square^{\mathrm{ind}}(\mu, \lambda)$-sequence.
\end{definition}

\begin{remark} \label{indexed_square_remark}
  These indexed square principles follow from more familiar non-indexed square principles.
  For example, if $\mu$ is singular, then $\square_\mu$ implies
  $\square^{\mathrm{ind}}_{\mu, \cf(\mu)}$ (cf.\ \cite{lh_aronszajn_trees}),
  while, if $\lambda < \mu$ are regular infinite cardinals, then
  $\square(\mu)$ implies $\square^{\mathrm{ind}}(\mu, \lambda)$
  (cf.\ \cite{lh_lucke}).
\end{remark}

\begin{theorem} \label{bounded_order_negative_relation}
  Suppose that $\mu$ is a singular cardinal and $\square_\mu$ holds. Then
  \[
    \mu^+ \not\rightarrow_{wc} [\mu]^2_{\cf(\mu), < \cf(\mu)}.
  \]
\end{theorem}

\begin{proof}
  By Remark \ref{indexed_square_remark}, $\square_\mu$ implies
  $\square^{\mathrm{ind}}_{\mu, \cf(\mu)}$. Therefore, we may fix a
  $\square^{\mathrm{ind}}_{\mu, \cf(\mu)}$-sequence $\langle C_{\alpha, i} \mid
  \alpha \in \acc(\mu^+), ~ i(\alpha) \leq i < \cf(\mu) \rangle$ and a
  sequence $\langle \mu_i \mid i < \cf(\mu) \rangle$ of regular cardinals
  such that
  \begin{itemize}
    \item $\sup(\{\mu_i \mid i < \cf(\mu)\}) = \mu$;
    \item for all $\alpha \in \acc(\mu^+)$ and all $i(\alpha) \leq i < \cf(\mu)$,
    we have $|C_{\alpha, i}| < \mu_i$.
  \end{itemize}
  Define a coloring $c:[\acc(\mu^+)]^2 \rightarrow \cf(\mu)$ by letting $c({\alpha,
  \beta})$ be the least ordinal $i < \cf(\mu)$ such that $i(\beta) \leq i$ and
  $\alpha \in \acc(C_{\beta, i})$ for all $\alpha < \beta$ in $\acc(\mu^+)$.

  We claim that $c$ witnesses the negative relation $\mu^+ \not\rightarrow_{wc}
  [\mu]^2_{\cf(\mu), < \cf(\mu)}$. (More formally, the composition of $c$ with the unique
  order-preserving bijection from $\mu^+$ to $\acc(\mu^+)$ will witness
  the negative relation.) To show this, we will prove that if $i < \cf(\mu)$ is
  a color, $\Lambda \subseteq i$, and $X \subseteq \acc(\mu^+)$ is well-connected in
  $\Lambda$, then $\otp(X) \leq \mu_i$.

  To this end, fix a color $i < \cf(\mu)$, a set $\Lambda \subseteq i$, and a set
  $X \subseteq \acc(\mu^+)$ that is well-connected in $\Lambda$.

  \begin{claim}
    For all $\alpha < \beta$ in $X$, we have $\alpha \in \acc(C_{\beta, i})$.
  \end{claim}

  \begin{proof}
    The proof will be by induction on the minimal length of a path connecting
    $\alpha$ and $\beta$ as in Definition \ref{well_connected_set_def}.

    Fix $\alpha < \beta$ in $X$. Since $X$ is well-connected in $\Lambda$, we
    can fix a path $\vec{\alpha} = \langle \alpha_0, \ldots, \alpha_n \rangle$ such that
    \begin{itemize}
      \item $\alpha_0 = \alpha$ and $\alpha_n = \beta$;
      \item $\alpha_k \geq \alpha$ for all $k \leq n$; and
      \item $c(\alpha_k, \alpha_{k + 1}) \in \Lambda$ for all $k < n$.
    \end{itemize}
    Assume moreover that $\vec{\alpha}$ has minimal length among all such paths.
    If $n = 1$, then we will have $c(\alpha, \beta) \in \Lambda$. Therefore,
    there is $j \in \Lambda$ such that $\alpha \in \acc(C_{\beta, j}) \subseteq
    \acc(C_{\beta, i})$, so $\alpha \in \acc(C_{\beta, i})$, as desired.

    Suppose therefore that $n > 1$ and that we have established the claim for
    all pairs connected by paths of length less than $n$. In particular,
    it follows that $\alpha \in \acc(C_{\alpha_{n-1}, i})$. Note also that
    $c(\alpha_{n-1}, \beta) \in \Lambda$, so either $\alpha_{n-1} \in
    \acc(C_{\beta, i})$ or $\beta \in \acc(C_{\alpha_{n-1}}, i)$. 
    Let $\gamma = \min(\{\alpha_{n-1}, \beta\})$.
    Then $C_{\alpha_{n-1}, i} \cap \gamma = C_{\beta, i} \cap \gamma$, so, since
    $\alpha \in \acc(C_{\alpha_{n-1}, i})$, it follows that $\alpha \in
    \acc(C_{\beta, i})$, as desired.
  \end{proof}

  Now suppose for sake of contradiction that $\otp(X) > \mu_i$. It follows
  that there is $\beta \in X$ such that $\otp(X \cap \beta) = \mu_i$. But
  then, by the claim, we have $X \cap \beta \subseteq \acc(C_{\beta, i})$,
  contradicting the fact that $|C_{\beta, i}| < \mu_i$. Therefore, for
  every color $i < \cf(\mu)$, every $\Lambda \subseteq i$,
  and every set $X \subseteq \acc(\mu^+)$ that
  is well-connected in $\Lambda$, we have $|X| < \mu$, and hence $c$ witnesses
  $\mu^+ \not\rightarrow_{wc} [\mu]^2_{\cf(\mu), < \cf(\mu)}$.
\end{proof}

\begin{theorem} \label{unbounded_order_negative_relation}
  Suppose that $\lambda < \mu$ are infinite regular cardinals and
  $\square(\mu)$ holds. Then
  \[
    \mu \not\rightarrow_{wc} [\mu]^2_{\lambda, < \lambda}.
  \]
\end{theorem}

\begin{proof}
  The proof is quite similar to that of Theorem \ref{bounded_order_negative_relation},
  so we only indicate its differences. By Remark \ref{indexed_square_remark},
  we can fix a $\square^{\mathrm{ind}}(\mu, \lambda)$-sequence $\langle C_{\alpha, i} \mid
  \alpha \in \acc(\mu), ~ i(\alpha) \leq i < \lambda \rangle$. Define a coloring
  $c : [\acc(\mu)]^2 \rightarrow \lambda$ by letting $c(\alpha, \beta)$ be the least
  ordinal $i < \lambda$ such that $i(\beta) \leq i$ and $\alpha \in \acc(C_{\beta, i})$.

  We claim that $c$ witnesses the negative partition relation in the statement of
  the theorem.
  Suppose that $i < \lambda$, $\Lambda \subseteq i$, and $X \subseteq \acc(\mu)$
  is well-connected in $\Lambda$. Exactly as in the proof of Theorem
  \ref{bounded_order_negative_relation}, we can prove that, for all
  $\alpha < \beta$ in $X$, we have $\alpha \in \acc(C_{\beta, i})$. Now suppose
  for sake of contradiction that $|X| = \mu$, and let $D = \bigcup_{\alpha \in X}
  C_{\alpha, i}$. Since $C_{\beta, i}$ end-extends $C_{\alpha, i}$ for all
  $\alpha < \beta$ in $X$, it follows that $D$ is a club in $\mu$.

  We claim that $D \cap \alpha = C_{\alpha, i}$ for all $\alpha \in \acc(D)$.
  Indeed, if $\alpha \in \acc(D)$, then, letting $\beta = \min(X \setminus (\alpha + 1))$,
  we have $D \cap \beta = C_{\beta, i}$, so $D \cap \alpha = C_{\beta, i} \cap
  \alpha$. Since $\alpha \in \acc(D)$, it follows that $\alpha \in \acc(C_{\beta, i})$,
  so, by Definition \ref{indexed_square_mu_lambda_def}, we have
  $C_{\alpha, i} = C_{\beta, i} \cap \alpha = D \cap \alpha$.

  $D$ is then a counterexample to clause (6) of Definition \ref{indexed_square_mu_lambda_def},
  so it follows that $|X| < \mu$. Therefore, $c$ witnesses
  $\mu \not\rightarrow_{wc} [\mu]^2_{\lambda, < \lambda}$.
\end{proof}

Recall that, if $\mu$ is a regular uncountable cardinal and $\square(\mu)$ fails,
then $\mu$ is weakly compact in $\mathrm{L}$. As a result, we immediately obtain
the following equiconsistency.

\begin{corollary} \label{wkly_compact_cor}
  The following statements are equiconsistent over $\mathrm{ZFC}$.
  \begin{enumerate}
    \item There exists a weakly compact cardinal.
    \item There exist infinite regular cardinals $\lambda < \mu$ such that
    $\mu \rightarrow_{wc} [\mu]^2_{\lambda, <\lambda}$ holds.
  \end{enumerate}
\end{corollary}

\section{A sharp positive result at the continuum} \label{forcing_sec}

In \cite[Proposition 8]{highly_connected}, it is shown that, for all infinite
cardinals $\mu$ and $\lambda$ with $\mu \leq 2^\lambda$, we have the negative
relation
\[
  \mu \not\rightarrow_{hc} (\mu)^2_\lambda.
\]
In particular,
\[
  2^\lambda \not\rightarrow_{hc} (2^\lambda)^2_\lambda.
\]
In \cite{highly_connected}, this was seen to be sharp in the sense that reducing
the number of colors results in a consistent statement.
In particular, it was shown that assuming the consistency of a weakly compact
cardinal, it is consistent that,
for example, the positive relation $2^{\aleph_1} \rightarrow_{hc}
(2^{\aleph_1})^2_{\aleph_0}$ holds.

In this section, we show it is sharp in a different sense, namely that
increasing the number of colors allowed to appear in the desired homogeneous
set also results in a consistent statement.
More precisely, we show that assuming the consistency of a weakly compact
cardinal, the positive relation
$2^\lambda \rightarrow_{hc} [2^\lambda]^2_{\lambda, 2}$ consistently holds.

\begin{theorem} \label{positive_square_bracket_relation}
  Suppose that $\lambda < \theta$ are infinite regular cardinals and
  $\theta$ is weakly compact. Let $\bb{P}$ be the forcing to add
  $\theta$-many Cohen subsets to $\lambda$. Then, after forcing with
  $\bb{P}$, we have
  \[
    2^\lambda \rightarrow_{hc} [2^\lambda]^2_{\lambda', 2}
  \]
  for all $\lambda' < 2^\lambda$.
\end{theorem}

\begin{proof}
  Elements of $\bb{P}$ are partial functions $p: \theta \rightarrow 2$ such that
  $|\dom(p)| < \lambda$, ordered by reverse inclusion.
  Note that, in $V^\bb{P}$, we have $2^\lambda = \theta$, so we must
  show that $\theta \rightarrow_{hc} [\theta]^2_{\lambda', 2}$ holds
  after forcing with $\bb{P}$ for all $\lambda' < \theta$.
  Fix a cardinal $\lambda' < \theta$, a condition $p \in \bb{P}$
  and a $\bb{P}$-name $\dot{c}$ that is
  forced by $p$ to be a name for a function from $[\theta]^2$ to
  $\lambda'$. We will find a condition $r \leq p$ and colors
  $i_0, i_1 < \lambda'$ such that $r$ forces the existence of
  a highly connected subgraph $(\dot{X}, \dot{E})$ of $(\theta, [\theta]^2)$
  such that $\dot{c}``\dot{E} = \{i_0, i_1\}$.

  For all $\alpha < \beta < \theta$, fix a condition $q_{\alpha, \beta} \leq p$
  and a color $i_{\alpha, \beta} < \lambda'$ such that
  \[
    q_{\alpha, \beta} \Vdash ``\dot{c}(\alpha, \beta) = i_{\alpha, \beta}".
  \]
  Without loss of generality, let us assume that $\{\alpha, \beta\} \subseteq
  \dom(q_{\alpha, \beta})$. We will take advantage of the weak compactness of
  $\theta$ to find an unbounded $A \subseteq \theta$ such that the conditions
  $\{q_{\alpha, \beta} \mid (\alpha, \beta) \in [A]^2\}$ enjoy a certain
  uniformity. For notational convenience, for all $(\alpha, \beta) \in [A]^2$,
  let $u_{\alpha, \beta} = \dom(q_{\alpha, \beta})$.

  First, we can appeal to the weak compactness of $\theta$ to
  find ordinals $i_0 < \lambda'$ and $\xi_0 < \lambda$, a function
  $d_0 : \xi_0 \rightarrow 2$, and an unbounded $A_0 \subseteq
  \theta$ such that, for all $(\alpha, \beta) \in [A_0]^2$, we have
  \begin{enumerate}
    \item $i_{\alpha, \beta} = i_0$;
    \item $\otp(u_{\alpha, \beta}) = \xi_0$;
    \item letting $u_{\alpha, \beta}$ be enumerated in increasing
    order as $\{\gamma_\zeta \mid \zeta < \xi_0\}$, we have that
    $q_{\alpha, \beta}(\gamma_\zeta) = d_0(\zeta)$ for all $\zeta < \xi_0$.
  \end{enumerate}

  We now appeal to \cite[Lemma 18]{highly_connected} (see also
  \cite{todorcevic_reals_and_positive_partition_relations}), which can be seen as a
  two-dimensional $\Delta$-system lemma, to find an unbounded $A_1 \subseteq A_0$
  such that
  \begin{enumerate}
    \setcounter{enumi}{3}
    \item for all $\alpha \in A_1$, the set $\{u_{\alpha, \beta} \mid
    \beta \in A_1 \setminus (\alpha + 1)\}$ is a $\Delta$-system, with
    root $u_\alpha^+$;
    \item letting $A_1^* = A_1 \setminus \{\min(A_1)\}$, for all
    $\beta \in A_1^*$, the set
    $\{u_{\alpha, \beta} \mid \alpha \in A_1 \cap \beta\}$ is a
    $\Delta$-system, with root $u_\beta^-$;
    \item the sets $\{u_\alpha^+ \mid \alpha \in A_1\}$ and
    $\{u_\alpha^- \mid \alpha \in A_1^*\}$ are both $\Delta$-systems, with
    roots $u^+_\emptyset$ and $u^-_\emptyset$, respectively.
  \end{enumerate}
  It is in fact easy to see, given the above discussion, that the roots
  $u^+_\emptyset$ and $u^-_\emptyset$ of item (6) must both be equal
  to the set
  \[
    u_\emptyset := \bigcap_{(\alpha, \beta) \in [A_1]^2} u_{\alpha, \beta}.
  \]
  By thinning out $A_1$ further, using the weak compactness of $\theta$, we can
  assume that:
  \begin{enumerate}
    \setcounter{enumi}{6}
    \item For all $(\alpha, \beta) \in [A_1^*]^2$, we have $\otp(u_{\alpha}^+)
    = \otp(u_{\beta}^+)$ and $\otp(u_\alpha^-) = \otp(u_\beta^-)$.
    \item For all $(\alpha, \beta), (\gamma, \delta) \in [A_1^*]^2$, we have that
    $u_\alpha^+$, $u_\beta^-$, and $u_\emptyset$ ``sit inside" of $u_{\alpha, \beta}$
    in the same way that $u_\gamma^+$, $u_\delta^-$, and $u_\emptyset$ ``sit inside"
    of $u_{\gamma, \delta}$. More formally, we have
    \begin{enumerate}
      \item $\{\zeta < \xi_0 \mid u_{\alpha, \beta}(\zeta) \in u_\alpha^+\} =
      \{\zeta < \xi_0 \mid u_{\gamma, \delta}(\zeta) \in u_\gamma^+\}$;
      \item $\{\zeta < \xi_0 \mid u_{\alpha, \beta}(\zeta) \in u_\beta^-\} =
      \{\zeta < \xi_0 \mid u_{\gamma, \delta}(\zeta) \in u_\delta^-\}$;
      \item $\{\zeta < \xi_0 \mid u_{\alpha, \beta}(\zeta) \in u_\emptyset\} =
      \{\zeta < \xi_0 \mid u_{\gamma, \delta}(\zeta) \in u_\emptyset\}$.
    \end{enumerate}
  \end{enumerate}

  By items (3) and (8) above, it follows that, if $\alpha < \beta < \gamma$
  are all elements of $A_1^*$, then
  \begin{itemize}
    \item $q_{\alpha, \beta}$ and $q_{\alpha, \gamma}$ are compatible in $\bb{P}$;
    \item $q_{\alpha, \gamma}$ and $q_{\beta, \gamma}$ are compatible in $\bb{P}$.
  \end{itemize}

  It also follows that, for all $\alpha \in A^*_1$, we can define a condition
  $q^+_\alpha$ by letting $q^+_\alpha = q_{\alpha, \beta}
  \restriction u^+_\alpha$ for some $\beta \in A_1 \setminus (\alpha + 1)$, and
  that this definition is independent of our choice of $\beta$. Similarly,
  for all $\beta \in A_1^*$, we can define $q^-_\beta$ by letting
  $q^-_\beta = q_{\alpha, \beta} \restriction u^-_\beta$ for some
  $\alpha \in A_1 \cap \beta$, and we can define $q_\emptyset$ by
  letting $q_\emptyset = q_{\alpha, \beta} \restriction u_\emptyset$ for
  some $(\alpha, \beta) \in [A_1]^2$. Again, these definitions are
  independent of our choices.

  \begin{claim}
    There is an unbounded $A_1^{**} \subseteq A_1^*$ such that, for all
    $(\alpha, \beta) \in [A_1^{**}]^2$, we have $u_\alpha^- \cap u_\beta^+ =
    u_\emptyset$.
  \end{claim}

  \begin{proof}
    It is clear from the definition of $u_\emptyset$ that $u_\alpha^- \cap
    u_\beta^+ \supseteq u_\emptyset$ for all $(\alpha, \beta) \in [A_1^*]^2$.
    Define a function $f:[A_1^*]^2 \rightarrow 2$ by letting
    $f(\alpha, \beta) = 0$ if $u_\alpha^- \cap u_\beta^+ = u_\emptyset$ and
    letting $f(\alpha, \beta) = 1$ otherwise. By the weak compactness of
    $\theta$, we can find an unbounded set $A_1^{**} \subseteq A_1^*$ such that
    $f$ is constant on $[A_1^{**}]^2$. We claim that $f``[A_1^{**}]^2 = \{0\}$.
    It suffices to find a single pair $(\alpha, \beta) \in [A_1^{**}]^2$ such
    that $f(\alpha, \beta) = 0$. Let $\alpha = \min(A_1^{**})$. We know that
    $\{u_{\beta}^+ \mid \beta \in A_1^{**}\}$ forms a $\Delta$-system with
    root $u_\emptyset$. Since $|u_{\alpha}^-| < \lambda < \theta$, we can
    find $\beta \in A_1^{**} \setminus (\alpha + 1)$ such that
    $u_{\beta}^+ \setminus u_\emptyset$ is disjoint from $u_{\alpha}^-$. But
    then $u_\alpha^- \cap u_\beta^+ = u_\emptyset$, so $f(\alpha, \beta) = 0$.
    It follows that $A_1^{**}$ is as desired.
  \end{proof}

  Fix an unbounded $A_1^{**} \subseteq A_1^*$ as given by the previous claim.
  It follows that, for all $(\alpha, \beta) \in
  [A_1^{**}]^2$, $q_\alpha^- \cup q_\beta^+$
  is a condition in $\bb{P}$ extending both $q_\alpha^-$ and $q_\beta^+$.
  For $(\alpha, \beta) \in [A_1^{**}]^2$, fix a condition $r_{\alpha, \beta}
  \leq q_\alpha^- \cup q_\beta^+$ and a color $j_{\alpha, \beta} < \lambda'$
  such that
  \[
    r_{\alpha, \beta} \Vdash ``\dot{c}(\alpha, \beta) = j_{\alpha, \beta}".
  \]
  Let $v_{\alpha, \beta} = \dom(r_{\alpha, \beta})$. Repeating the above process
  with $\langle r_{\alpha, \beta} \mid (\alpha, \beta) \in [A_1^{**}]^2 \rangle$
  in place of $\langle q_{\alpha, \beta} \mid (\alpha, \beta) \in [\theta]^2
  \rangle$, we can find ordinals $i_1 < \lambda'$ and $\xi_1 < \lambda$, a
  function $d_1:\xi_1 \rightarrow 2$, and an unbounded $A_2 \subseteq A_1^{**}$ such that
  \begin{enumerate}
    \setcounter{enumi}{8}
    \item for all $(\alpha, \beta) \in [A_2]^2$, we have
    \begin{enumerate}
      \item $j_{\alpha, \beta} = i_1$;
      \item $\otp(v_{\alpha, \beta}) = \xi_1$;
      \item letting $v_{\alpha, \beta}$ be enumerated in increasing order as
      $\{\gamma_\zeta \mid \zeta < \xi_1\}$, we have that $r_{\alpha, \beta}
      (\gamma_\zeta) = d_1(\zeta)$ for all $\zeta < \xi_1$;
    \end{enumerate}
    \item for all $\alpha \in A_2$, the set $\{v_{\alpha, \beta} \mid
    \beta \in A_2 \setminus (\alpha + 1)\}$ is a $\Delta$-system, with root
    $v_\alpha$;
    \item the set $\{v_\alpha \mid \alpha \in A_2\}$ is a $\Delta$-system, with
    root $v_\emptyset$;
    \item for all $(\alpha, \beta)$, $(\gamma, \delta) \in [A_2]^2$, the sets
    $v_\emptyset$ and $v_\alpha$ ``sit inside" $v_{\alpha, \beta}$ in the same
    way that the sets $v_\emptyset$ and $v_\gamma$ ``sit inside" $v_{\gamma, \delta}$.
  \end{enumerate}
  As above, we define conditions $\langle r_\alpha \mid \alpha \in A_2 \rangle$
  by letting $r_\alpha = r_{\alpha, \beta} \restriction v_\alpha$ for some
  $\beta \in A_2 \setminus (\alpha + 1)$, and we define $r_\emptyset =
  r_{\alpha, \beta} \restriction v_\emptyset$ for some $(\alpha, \beta)
  \in [A_2]^2$. Again, these definitions are independent of our choices.

  We thin out our set $A_2$ one final time in the following way. Define a
  function $e:[A_2]^2 \rightarrow 2$ by letting $e(\alpha, \beta) = 0$ if
  $q_{\alpha, \beta}$ and $r_\beta$ are compatible in $\bb{P}$, and letting
  $e(\alpha, \beta) = 1$ otherwise. Using the weak compactness of $\theta$,
  find an unbounded $A \subseteq A_2$ such that $e$ is constant on $[A]^2$.

  \begin{claim} \label{e_homogeneous_claim}
    $e``[A]^2 = \{0\}$.
  \end{claim}

  \begin{proof}
    Since $e$ is constant on $[A]^2$, it suffices to find a single pair
    $(\alpha, \beta) \in [A]^2$ such that $e(\alpha, \beta) = 0$.

    Fix $\beta \in A$ such that $|A \cap \beta| \geq \lambda$.
    The set $\{u_{\alpha, \beta} \mid \alpha \in A \cap \beta\}$ forms a
    $\Delta$-system with root $u_\beta^-$. Moreover, by construction, we have
    $v_\beta \supseteq u_\beta^-$ and $r_\beta \leq q_\beta^-$. We can
    therefore find $\alpha \in A \cap \beta$ such that $(u_{\alpha, \beta}
    \setminus u_\beta^-) \cap v_\beta = \emptyset$. Then we have
    $q_{\alpha, \beta} \restriction u_\beta^- = q_\beta^- \geq r_\beta$, so
    $q_{\alpha, \beta}$ and $r_\beta$ are compatible in $\bb{P}$ and hence
    $e(\alpha, \beta) = 0$.
  \end{proof}

  Now make the following assignments:
  \begin{itemize}
    \item $r = r_\emptyset$;
    \item $\dot{G}$ is the canonical $\bb{P}$-name for the $\bb{P}$-generic filter;
    \item $\dot{X}_0$ is a $\bb{P}$-name for the set $\{\alpha \in A \mid
    q_\alpha^+ \in \dot{G}\}$;
    \item $\dot{X}_1$ is a $\bb{P}$-name for the set $\{\beta \in A \mid
    r_\beta \in \dot{G}\}$;
    \item $\dot{X}$ is a $\bb{P}$-name for $\dot{X}_0 \cup \dot{X}_1$;
    \item $\dot{E}$ is a $\bb{P}$-name for $\dot{c}^{-1}(\{i_0, i_1\}) \cap
    [\dot{X}]^2$.
  \end{itemize}
  Notice that each $q_{\alpha, \beta} \leq p$ so also $q_\emptyset \leq p$. It
  similarly follows that $r \leq p$. We will end the proof by
  showing that $r$ forces $(\dot{X}, \dot{E})$ to be a highly connected
  graph of cardinality $\theta$.

\begin{claim} \label{large_x_claim}
  $r \Vdash ``|\dot{X}| = \theta"$.
\end{claim}

\begin{proof}
  We will show that $r \Vdash ``|\dot{X}_0| = \theta"$, which suffices. A similar proof will
  in fact show that $r \Vdash ``|\dot{X}_1| = \theta"$, as well.

  Fix an arbitrary condition $s \leq r$ and an $\eta < \theta$. It suffices to
  find $\alpha \in A \setminus \eta$ such that
  $q_\alpha^+$ and $s$ are compatible in $\bb{P}$. Since $s \leq r$, we have
  $\dom(s) \supseteq u_\emptyset$, and $s \leq q_\emptyset$. Recall that the
  set $\{u^+_\alpha \mid \alpha \in A \setminus \eta\}$ is a $\Delta$-system
  with root $u_\emptyset$. We can therefore find $\alpha \in A$ such that
  $(u_\alpha^+ \setminus u_\emptyset) \cap \dom(s) = \emptyset$. Since
  $q_\alpha^+ \restriction u_\emptyset = q_\emptyset \geq s$, it follows that
  $q_\alpha^+$ and $s$ are compatible in $\bb{P}$, as desired.
\end{proof}

To show that $r$ forces $(\dot{X}, \dot{E})$ to be highly connected, we first
establish a couple of preliminary claims.

\begin{claim} \label{x_0_extension_claim}
  $r \Vdash ``\forall (\alpha, \beta) \in [\dot{X}_0]^2 ~ \forall \eta < \theta
  ~ \exists \gamma \in
  \dot{X}_1 \setminus \eta ~ [\{q_{\alpha, \gamma}, q_{\beta, \gamma}\}
  \subseteq \dot{G}]"$.
\end{claim}

\begin{proof}
  Fix an ordinal $\eta < \theta$, a condition $s \leq r$
  and $(\alpha, \beta) \in [A]^2$ such that $s$ forces both
  $\alpha$ and $\beta$ to be in $\dot{X}_0$. Without loss of generality, we can
  assume that $s \leq q_\alpha^+$ and $s \leq q_\beta^+$ and that $\eta > \beta$.
  It will suffice
  to find $\gamma \in A \setminus \eta$ such that the conditions
  $s$, $q_{\alpha, \gamma}$, $q_{\beta, \gamma}$, and $r_\gamma$ are all
  pairwise compatible, since then the union of these four conditions would
  itself be a condition extending $s$ and forcing $\gamma$ to be as desired.
  Note also that, for all $\gamma \in A \setminus \eta$, we know
  that $q_{\alpha, \gamma}$ and $q_{\beta, \gamma}$ are compatible by the discussion
  following item (8) above.
  We also know that $r_\gamma$ is compatible with each of $q_{\alpha, \gamma}$
  and $q_{\beta, \gamma}$ by Claim \ref{e_homogeneous_claim}. It therefore
  suffices to find $\gamma \in A \setminus \eta$ such that $s$
  is compatible with each of $q_{\alpha, \gamma}$, $q_{\beta, \gamma}$, and
  $r_\gamma$.

  By assumption, we know that $\dom(s) \supseteq u_\alpha^+ \cup u_\beta^+ \cup v_\emptyset$ and
  that $s$ extends each of $q_\alpha^+$, $q_\beta^+$, and $r_\emptyset$.
  Recall also that the sets $\{u_{\alpha, \gamma} \mid \gamma \in A \setminus
  \eta\}$, $\{u_{\beta, \gamma} \mid \gamma \in A \setminus \eta\}$,
  and $\{v_\gamma \mid \gamma \in A \setminus \eta \}$ are $\Delta$-systems
  with roots $u_\alpha^+$, $u_\beta^+$, and $v_\emptyset$, respectively.
  We can therefore find $\gamma \in A \setminus \eta$ such that each of the sets
  $(u_{\alpha, \gamma} \setminus u_\alpha^+)$, $(u_{\beta, \gamma} \setminus
  u_\beta^+)$, and $v_\gamma \setminus v_\emptyset$ is disjoint from $\dom(s)$.
  But then we have
  \begin{itemize}
    \item $q_{\alpha, \gamma} \restriction \dom(s) = q_\alpha^+$;
    \item $q_{\beta, \gamma} \restriction \dom(s) = q_\beta^+$; and
    \item $r_\gamma \restriction \dom(s) = r_\emptyset$.
  \end{itemize}
  Therefore, since $s$ extends each of $q_\alpha^+$, $q_\beta^+$, and
  $r_\emptyset$, it follows that $s$ is compatible with each of
  $q_{\alpha, \gamma}$, $q_{\beta, \gamma}$, and $r_\gamma$, as desired.
\end{proof}

\begin{claim} \label{x_1_extension_claim}
  $r \Vdash ``\forall \alpha \in \dot{X}_1 ~ \forall \eta < \lambda
  ~ \exists \beta \in \dot{X}_0
  \setminus \eta ~ [r_{\alpha, \beta} \in \dot{G}]"$.
\end{claim}

\begin{proof}
  Fix an ordinal $\eta < \theta$, a condition $s \leq r$ and
  $\alpha \in A$ such that $s \Vdash ``\alpha \in \dot{X}_1"$.
  Without loss of generality, assume that $\eta > \alpha$ and $s \leq r_\alpha$,
  and hence $\dom(s) \supseteq v_\alpha$. The set $\{v_{\alpha, \beta} \mid
  \beta \in A \setminus \eta \}$ is a $\Delta$-system with root $v_\alpha$;
  we can therefore find $\beta \in A \setminus \eta$ such that
  $(v_{\alpha, \beta} \setminus v_\alpha) \cap \dom(s) = \emptyset$. We know that
  $r_{\alpha, \beta} \restriction \dom(s) = r_\alpha \geq s$, so it follows that
  $s$ and $r_{\alpha, \beta}$ are compatible. Recall that $r_{\alpha, \beta}
  \leq q_\beta^+$ and therefore forces $\beta$ to be in $\dot{X}_0$. Therefore,
  the union of $s$ and $r_{\alpha, \beta}$ forces $\beta$ to be as desired.
\end{proof}

Let $G$ be a $\bb{P}$-generic filter over $V$ with $r \in G$. Let $c$, $X_0$, $X_1$,
$X$, and $E$ be the realizations of $\dot{c}$, $\dot{X}_0$, $\dot{X}_1$,
$\dot{X}$, and $\dot{E}$, respectively, in
$V[G]$. By the definition of $\dot{E}$, we know that $c``[E]^2 \subseteq
\{i_0, i_1\}$. By Claim \ref{large_x_claim}, we know that $|X| = \theta$. It
thus remains to show that, for all $Y \in [X]^{<\theta}$, the graph
$(X \setminus Y, ~ E \cap [X \setminus Y]^2)$ is connected.

Fix $Y \in [X]^{<\theta}$, and let $Z = X \setminus Y$. Also fix
$(\alpha, \beta) \in [Z]^2$. Since $\theta$ is regular and $|Y| < |X| = \theta$,
there is $\eta < \theta$ such that $Y \subseteq \eta$, and hence
$X \setminus \eta \subseteq Z$. By increasing $\eta$ if necessarily, we may
assume that $\beta < \eta$. There are now a number of cases, not
necessarily mutually exclusive, to consider.

\textbf{Case 1: $\alpha, \beta \in X_0$.} In this case, Claim \ref{x_0_extension_claim}
implies that there is $\gamma \in X_1 \setminus \eta$ such that $\{q_{\alpha,
\gamma}, q_{\beta, \gamma}\} \subseteq G$. It follows that $c(\alpha, \gamma)
= c(\alpha, \gamma) = i_0$, so $\langle \alpha, \gamma, \beta \rangle$ is a
path from $\alpha$ to $\beta$ in $(Z, E \cap [Z]^2)$.

\textbf{Case 2: $\alpha \in X_0$ and $\beta \in X_1$.} By Claim
\ref{x_1_extension_claim}, we can find $\gamma \in X_0 \setminus \eta$
such that $r_{\beta, \gamma} \in G$. Then, by Claim \ref{x_0_extension_claim},
we can find $\delta \in X_1 \setminus (\gamma + 1)$ such that
$\{q_{\alpha, \delta}, q_{\gamma, \delta}\} \subseteq G$. It follows
that $c(\alpha, \delta) = c(\gamma, \delta) = i_0$ and
$c(\beta, \gamma) = i_1$, so $\langle \alpha, \delta, \gamma, \beta \rangle$
is a path from $\alpha$ to $\beta$ in $(Z, E \cap [Z]^2)$.

\textbf{Case 3: $\alpha \in X_1$ and $\beta \in X_0$.} This is symmetric to
Case 2.

\textbf{Case 4: $\alpha, \beta \in X_1$.} By Claim \ref{x_1_extension_claim},
we can first find $\gamma \in X_0 \setminus \eta$ such that $r_{\alpha, \gamma}
\in G$ and then $\delta \in X_0 \setminus (\gamma + 1)$ such that
$r_{\beta, \delta} \in G$. Then, by Claim \ref{x_0_extension_claim},
we can find $\epsilon \in X_1 \setminus (\delta + 1)$ such that
$\{q_{\gamma, \epsilon}, q_{\delta, \epsilon}\} \subseteq G$. It follows
that $c(\alpha, \gamma) = c(\beta, \delta) = i_1$ and $c(\gamma, \epsilon) =
c(\delta, \epsilon) = i_0$, so $\langle \alpha, \gamma, \epsilon, \delta,
\beta \rangle$ is a path from $\alpha$ to $\beta$ in $(Z, E \cap [Z]^2)$.

This exhausts all possible cases, so we have shown that, in $V[G]$,
$(X, E)$ is highly connected, thus finishing the proof.
\end{proof}

\begin{remark}
  We have seen that Theorem \ref{positive_square_bracket_relation} is sharp in
  the sense that the conclusion cannot be improved to $2^\lambda \rightarrow_{hc}
  (2^\lambda)^2_\lambda$. It is also sharp in the sense that the ``$hc$" subscript
  cannot be dropped. For example, it is easily seen that the coloring
  $\Delta : [{^\lambda} 2]^2 \rightarrow \lambda$ defined by letting
  $\Delta(f,g)$ be the least $i < \lambda$ such that $f(i) \neq g(i)$ witnesses
  the negative square bracket relation
  \[
    2^\lambda \not\rightarrow [\aleph_0]^2_{\lambda, {<}\aleph_0},
  \]
  and, more generally, if $\nu \leq \lambda$, $\mu \leq 2^\lambda$, and
  $2^\chi < \mu$ for all $\chi < \nu$, then $\Delta$ witnesses
  \[
    2^\lambda \not\rightarrow [\mu]^2_{\lambda, {<}\nu}.
  \]
\end{remark}

We end this section by noting the following equiconsistency that results from a 
special case of the preceding theorem.

\begin{corollary}
  The following statements are equiconsistent over ZFC.
  \begin{enumerate}
    \item There exists a weakly compact cardinal.
    \item $2^{\aleph_0} \rightarrow_{hc} [2^{\aleph_0}]^2_{\aleph_0, 2}$.
  \end{enumerate}
\end{corollary}

\begin{proof}
  The implication from the consistency of (1) to that of (2) follows directly from 
  Theorem \ref{positive_square_bracket_relation}. For the reverse implication, if 
  $\mu := \cf(2^{\aleph_0})$ is not weakly compact in L, then $\square(\mu)$ holds, 
  and hence we can fix a coloring $c:[\mu]^2 \rightarrow \omega$ witnessing 
  $\mu \not\rightarrow_{wc} [\mu]^2_{\aleph_0, {<}\aleph_0}$. Now let 
  $\langle \nu_\eta \mid \eta < \mu \rangle$ be an increasing sequence of ordinals, 
  cofinal in $2^{\aleph_0}$. For each $\alpha < 2^{\aleph_0}$, let 
  $\eta_\alpha$ be the least $\eta < \mu$ such that $\alpha \leq \nu_\eta$,
  and define a coloring $d:[2^{\aleph_0}]^2 \rightarrow \omega$ by setting 
  \[
    d(\alpha, \beta) = \begin{cases}
      c(\eta_\alpha, \eta_\beta) & \text{if } \eta_\alpha \neq \eta_\beta \\
      0 & \text{if } \eta_\alpha = \eta_\beta
    \end{cases}.
  \]
  Now if $X \in [2^{\aleph_0}]^{2^{\aleph_0}}$, $\Lambda \subseteq \omega$, 
  and $X$ is well-connected in $\Lambda$ with respect to $d$, then 
  $Y := \{\eta_\alpha \mid \alpha \in X\} \in [\mu]^\mu$ is well-connected in 
  $\Lambda$ with respect to $c$. Therefore, $d$ witnesses $2^{\aleph_0} 
  \not\rightarrow_{wc} [2^{\aleph_0}]^2_{\aleph_0, {<}\aleph_0}$ and hence, 
  \emph{a fortiori}, $2^{\aleph_0} \not\rightarrow_{hc} [2^{\aleph_0}]^2_{\aleph_0, 2}$.
\end{proof}

\section{Recent work} \label{questions_sec}

In the first draft of this paper, we included here the following two questions, which 
were open at the time, about whether certain positive instances of the highly connected 
partition relation are consistent (relative to the consistency of the existence of large 
cardinals). These questions were originally asked in \cite{highly_connected} and 
\cite{well_connected}, respectively.

\begin{question}[Bergfalk-Hru\v{s}\'{a}k-Shelah \cite{highly_connected}]
  Is $\aleph_2 \rightarrow_{hc} (\aleph_2)^2_{\aleph_0}$ consistent?
\end{question}

\begin{question}[Bergfalk \cite{well_connected}]
  Is $\aleph_{\omega + 1} \rightarrow_{hc} (\aleph_\omega)^2_{\aleph_0}$
  consistent? If so, what about $\aleph_{\omega + 1} \rightarrow_{hc}
  (\aleph_{\omega + 1})^2_{\aleph_0}$?
\end{question}

Both of these questions have positive answers. In a forthcoming work \cite{hsz}, 
Hru\v{s}\'{a}k, Shelah, and Zhang prove that, if $\kappa$ is measurable then, in the 
forcing extension by the L\'{e}vy collapse $\mathrm{Coll}(\omega_1, {<}\kappa)$, 
$\aleph_2 \rightarrow_{hc} (\aleph_2)^2_{\aleph_0}$ holds. They also prove that if, 
moreover, $\kappa$ is 
$\kappa^{+\omega+1}$-supercompact, then $\aleph_{\omega + 1} \rightarrow_{hc}
  (\aleph_{\omega + 1})^2_{\aleph_0}$ also holds in the extension by 
  $\mathrm{Coll}(\omega_1, {<}\kappa)$.

A more open-ended, speculative question involves generalizations of the
partition relations being studied to higher dimensions. In the case of
the classical partition relation $\nu \rightarrow (\mu)^2_\lambda$, it is
clear how to generalize to $\nu \rightarrow (\mu)^k_\lambda$ for
$k > 2$. In the case of $\nu \rightarrow_{hc} (\mu)^2_\lambda$ or
$\nu \rightarrow_{wc} (\mu)^2_\lambda$, however, such a generalization would
require isolating the correct definition(s) of ``highly connected" and
``well-connected" in the context of $k$-uniform hypergraphs. There are a number
of different approaches one might take to this generalization, but it is
presently unclear, at least to us, which, if any, of these approaches yields
an interesting theory of higher-dimensional partition relations. We therefore
ask the following deliberately vague problem.

\begin{problem} \label{problem_53}
  Isolate the correct definition(s) for a generalization (or generalizations)
  of highly connected or well-connected Ramsey theory to higher dimensions.
\end{problem}

In connection with Problem \ref{problem_53}, we want to highlight recent work of Bannister, 
Bergfalk, Moore, and Todorcevic \cite{bbmt} in which they introduce a partition 
hypothesis $\mathrm{PH}_n(\mathbb{P}, \lambda)$ for a given $0 < n < \omega$, a directed 
quasi-order $\mathbb{P}$, and a cardinal $\lambda$. $\mathrm{PH}_n(\mathbb{P}, \lambda)$ is 
a positive $(n+1)$-dimensional partition relation yielding information about colorings 
of the form $c:[\mathbb{P}]^{n+1} \rightarrow \lambda$. For cardinals $\nu$ and $\lambda$, the 
principle $\mathrm{PH}_1(\nu, \lambda)$ is related to the relations $\nu \rightarrow_{\mathrm{hc}} 
(\nu)^2_\lambda$ and $\nu \rightarrow_{\mathrm{wc}} (\nu)^2_\lambda$, and directly implies the 
latter, so this work can be seen in part as one possible approach to Problem \ref{problem_53}.
We feel that there remains much fruitful work to be done in further study of these and related 
partition principles.

\bibliographystyle{plain}
\bibliography{bib}

\end{document}